\documentclass[11pt]{amsart}
\usepackage[T1]{fontenc}
\usepackage{bbm}
\usepackage{amsmath, amssymb, amsfonts, amsthm, verbatim, dsfont, graphicx}
\usepackage[hidelinks]{hyperref}
\usepackage{mathtools}
\usepackage{array}
\usepackage{tabu}
\mathtoolsset{showonlyrefs}

\usepackage{graphicx, color}

\makeatletter
\renewcommand\tableofcontents{%
    \section*{\huge{Table of Contents}
        \@mkboth{%
           \MakeUppercase\contentsname}{\MakeUppercase\contentsname}}
    \@starttoc{toc}%
    } 
\makeatother

\newsavebox\MBox

\newtheorem{thm}{Theorem}
\newtheorem{lem}[thm]{Lemma}
\newtheorem{prop}[thm]{Proposition}

\newtheorem{cor}[thm]{Corollary}
\newtheorem{defn}[thm]{Definition}

\newtheorem{rmk}{Remark}

\numberwithin{rmk}{section}
\numberwithin{thm}{section}
\numberwithin{equation}{section}
\numberwithin{figure}{section}

\theoremstyle{plain} 
\newcommand{\thistheoremname}{}
\newtheorem{genericthm}[thm]{\thistheoremname}

  \newtheorem*{genericthm*}{\thistheoremname}
\newenvironment{namedthm*}[1]
  {\renewcommand{\thistheoremname}{#1}%
   \begin{genericthm*}}
  {\end{genericthm*}}

\newcommand{\cA}{\mathcal{A}}
\newcommand{\cB}{\mathcal{B}}

\newcommand{\cD}{\mathcal{D}}

\newcommand{\cH}{\mathcal{H}}

\newcommand{\cS}{\mathcal{S}}

\newcommand{\bN}{\mathbb{N}}

\newcommand{\R}{\mathbb{R}}

\def\frh{{\mathfrak h}}
\def\frH{{\mathfrak H}}

\newcommand{\W}{\textbf{W}}

\newcommand{\Tr}{\textup{Tr}}

\newcommand{\barphi}{\overline{\phi}}

\def\ux{{\underline{x}}}

\def\cS{{\mathcal S}}

\def\Tr{{\rm Tr}}

\def\tr{{\rm Tr}}

\title[An infinite sequence of conserved quantities for the cubic GP hierarchy]{An infinite sequence of conserved quantities for the cubic Gross-Pitaevskii hierarchy on $\mathbb{R}$}

\author[D. Mendelson]{Dana Mendelson$^1$}
\address{$^1$  
Department of Mathematics \\ 
University of Chicago\\  
5734 S. University Avenue \\ 
Chicago, IL  60637
}
\email{dana@math.uchicago.edu}
\thanks{$^1$ D.M. is funded in part by NSF DMS-1128155, and gratefully acknowledges support from the Institute for Advanced Study at Princeton.}

\author[A. Nahmod]{Andrea R. Nahmod$^2$}
\address{$^2$  
Department of Mathematics \\ University of Massachusetts\\  710 N. Pleasant Street, Amherst MA 01003}
\email{nahmod@math.umass.edu}
\thanks{$^2$ A.N. is funded in part by NSF DMS-1201443 and DMS-1463714.}

\author[N. Pavlovi\'c]{Nata\v{s}a Pavlovi\'c$^3$}
\address{$^3$  
Department of Mathematics\\ 
University of Texas at Austin\\ 
2515 Speedway, Stop C1200\\
Austin, TX 78712}
\email{natasa@math.utexas.edu}
\thanks{$^3$ N.P. is funded in part by NSF DMS-1516228.}

\author[G. Staffilani]{Gigliola Staffilani$^4$}
\address{$^4$ Department of Mathematics\\
Massachusetts Institute of Technology\\ 
77 Massachusetts Avenue,  Cambridge, MA 02139}
\email{gigliola@math.mit.edu}
\thanks{$^4$ G.S. is funded in part by NSF
 DMS-1362509 and DMS-1462401.}

\begin{document}
\maketitle
\begin{abstract}
We consider the cubic Gross-Pitaevskii (GP) hierarchy on $\mathbb{R}$, which is an infinite hierarchy of coupled linear inhomogeneous PDE appearing in the derivation of the cubic nonlinear Schr\"odinger equation from quantum many-particle systems. In this work, we identify an infinite sequence of operators which generate infinitely many conserved quantities for solutions of the GP hierarchy.
\end{abstract}

\section{Introduction}

In recent years, major advances have been made towards understanding the macroscopic properties of quantum mechanical systems with a very large number of particles\footnote[5]{Between $N \sim 10^3$ for very dilute Bose-Einstein samples, up to values of the order $N \sim 10^{30}$ in stars.}.  Mathematically, one considers an appropriate scaling limit, for example the limit of an $N$-particle system as $N\rightarrow \infty$, and derives a governing equation for the limiting system, which is then expected to be a good approximation for the macroscopic properties observed in experiments where the number of particles $N$  is very large, but finite, see for instance \cite{schlein_clay} for a nice survey, as well as references therein. 

The subject of our work is the one-dimensional cubic Gross-Pitaevskii (GP) hierarchy, which we define in  \eqref{equ:gp_intro}-\eqref{eq-def-b_intro} below. The dynamics of the GP hierarchy are closely connected to those of the one-dimensional cubic nonlinear Schr\"odinger (NLS) equation \eqref{equ:NLS_intro}, see for instance \cite{CHPS} and \S\ref{sec:gp_hierarchy} for more details. Nonetheless, although it is well-known that the cubic NLS on $\mathbb{R}$ is an example of an integrable equation, questions about an integrable structure for the complex dynamics of the GP hierarchy itself have yet to be explored.

A principle aim for this work is to begin to investigate whether, and if so how, the integrable structure of the one-dimensional cubic NLS manifests itself in the one-dimensional cubic GP hierarchy. In our main Theorem \ref{main}, we exhibit infinitely many conservation laws for the cubic GP hierarchy in 1D, which we believe is an important step towards understanding a possible integrable structure at the level of the GP hierarchy. Should such a structure be available, this could provide new information about the integrable structure of the cubic NLS. Indeed, we suspect that an integrable structure for the GP hierarchy would be a natural ingredient in eventually identifying an integrable structure at the level of the $N$-particle system from which the integrable structure for the cubic NLS could be derived.

\medskip
We now turn to the precise definition of the cubic GP hierarchy. The cubic GP hierarchy in one spatial dimension governs the evolution of 
a  sequence of functions $ \gamma^{(k)}(t , \ux_k, \ux_k')$, with $t \geq 0$, 
$\ux_k=(x_1,\dots,x_k) \in {\R}^k$, $\ux_k'=(x'_1,\dots,x'_k) \in {\R}^k$, 
arising from quantum many-particle systems as
the number of particles goes to infinity. The GP hierarchy is given by:

\begin{align}\label{equ:gp_intro}
	i\partial_t \gamma^{(k)} &= - \sum_{j=1}^k (\Delta_{x_j} - \Delta_{x'_j}) \gamma^{(k)}   
	\, + \,   2 \kappa B_{k+1} \gamma^{(k+1)}   \,, \;\;k\in\bN\,,
\end{align}
with $\kappa \in \{\pm 1 \}$ and 
\begin{align} 
\label{eq-def-b_intro}
	B_{k+1}\gamma^{(k+1)}
     = \sum_{j=1}^k B^+_{j;k+1 }\gamma^{(k+1)} - \sum_{j=1}^k B^-_{j;k+1 }\gamma^{(k+1)},
\end{align}
where
\begin{align}
\label{equ:collision_intro}
& \Bigl(B^+_{j;k+1}\gamma^{(k+1)}\Bigr)(t,x_1,\dots,x_k;x_1',\dots,x_k') \\
    & \quad \quad = \int dx_{k+1}  dx_{k+1}' \delta(x_{k+1}-x_{k+1}' )	\delta(x_j-x_{k+1})
        \gamma^{(k+1)}(t,x_1,\dots,x_{k+1};x_1',\dots,x_{k+1}') \nonumber
\end{align}
and similarly for $B^-_{j;k+1}$ with $\delta(x_j-x_{k+1})$ replaced by $\delta(x_j'-x_{k+1})$. The case $ \kappa = +1$ corresponds to the focusing hierarchy (which is related to the focusing NLS), while $\kappa = -1$ corresponds to the defocusing hierarchy. In the sequel, we will be working solely with the $B^{+}$ operators and not the $B^{-}$ operators, so we will abuse notation slightly and denote 
\[
B_{j;k}:= B_{j;k}^{+}. 
\]
This hierarchy consists of an infinite sequence of coupled linear equations which arise in the so called Gross-Pitaevskii limit
from systems of bosons interacting with two-body potentials. The GP hierarchy admits a special class of factorized solutions, given by
\begin{align}
\bigl(\gamma^{(k)}(t,\ux_k, \ux_k') \bigr)_{k \in \bN} = \bigl( | \phi \rangle \langle \phi|^{\otimes k}\bigr)_{k \in \bN} = \Bigl( \prod_{i=1}^k \phi(x_i) \overline{\phi}(x_i')\Bigr)_{k \in \bN}
\end{align}
for functions $\phi$ which solve the cubic NLS on $\R$:
\begin{align} \label{equ:NLS_intro}
i \phi_t + \partial_{xx} \phi = 2 \kappa  |\phi|^2 \phi
\end{align}
for $\kappa \in \{\pm 1\}$. The one-dimensional cubic NLS \eqref{equ:NLS_intro} has been extensively studied both for its role as a dispersive equation and as an example of an integrable model. This model can be solved exactly by the method of inverse scattering by Zakharov-Shabat \cite{ZS72, ZS}, and also \cite{AKNS}. There are many consequences of integrability for an equation, but perhaps one of the most well-known is the existence of a sequence of infinitely many integrals of motion.  The construction of these integrals of motion for the cubic NLS leads to the natural question of whether it is possible to generate infinitely many conserved quantities for smooth solutions of the GP hierarchy, even before knowing whether the GP hierarchy is integrable. In this paper we answer this question in the affirmative by proving the existence of infinitely many operators which generate conserved quantities for smooth solutions of the GP hierarchy. The existence of these conserved quantities for the GP hierarchy provides a first step towards understanding a possible integrable structure for the GP hierarchy, and perhaps, eventually, understanding a physical derivation of the integrable structure for the cubic NLS from many body quantum systems.

We also view this work from the perspective of the program of passing results from nonlinear PDE, such as the nonlinear Schr\"{o}dinger, to the infinite particle system given by the GP hierarchy. The quantum de Finetti theorem has proven to be a crucial tool for this type of analysis (see \S\ref{sec:quantum_de_finetti} for precise statements of this theorem). This approach has, for example, proven effective in establishing a new proof of uniqueness \cite{CHPS} and the first scattering result \cite{CHPS-sc} for the GP hierarchy.
  
To prove the existence of infinitely many operators which generate conserved quantities for smooth solutions of the GP hierarchy operators we proceed as follows. First we define operators $\{\W_{n+1}^j \}_{n \in \bN}$ (for certain $j \in \bN$) on factorized solutions of the GP hierarchy via a recursive definition. Our goal in this definition is to obtain conservation laws which are compatible with those for the cubic NLS. Subsequently we extend the definition to general solutions to the GP hierarchy using a recent uniqueness result of Hong-Taliaferro-Xie  \cite{HTX} for the GP hierarchy based on quantum de Finetti theorems  (see Theorem \ref{thm-strongDeFinetti} and Theorem \ref{thm-weakDeFinetti}), and some functional analytic tools on Bochner integrals. We record explicitly the first few operators $\W_n^j$ to illustrate their form:
\begin{align*}
\qquad \W_1^j &= Id_j \\
\W_2^j &= -i \partial_{x_j}\Tr_{j+1} \\
\W_3^j &= - \partial^2_{x_j}\Tr_{j+1j+2} + \kappa B_{j, j+1}\Tr_{j+2} \\
\W_4^j &= i \partial^3_{x_j}\Tr_{j+1, j+2,j+3} - i \kappa\partial_{x_j} B_{j, j+1}\Tr_{j+2} - i \kappa ( B_{j, j+1} \partial_{x_{j+1}}\Tr_{j+2} + B_{j, j+2} \partial_{x_{j}} \Tr_{j+1})\Tr_{j+3}\\
\W_5^j &= \partial^4_{x_j}\Tr_{j+2, j+3, j+4, j+5} -i \partial_{x_j} \Bigl[ i \kappa\partial_{x_j} B_{j, j+1}\Tr_{j+2} - i \kappa ( B_{j, j+1} \partial_{x_{j+1}}\Tr_{j+2} + B_{j, j+2} \partial_{x_{j}} \Tr_{j+1})\Tr_{j+3} \Bigr]\\
&+ \kappa \Bigl[ B_{j;j+1} (- \partial^2_{x_{j+1}}\Tr_{j+2, j+3} + \kappa B_{j+1, j+2}\Tr_{j+3}) \\
& \hspace{24mm} - B_{j;j+2} \partial_{x_j} \Tr_{j+1}\partial_{x_{j+2}}\Tr_{j+3} + B_{j;j+3}( - \partial^2_{x_j}\Tr_{j+1, j+2} + \kappa B_{j, j+1}\Tr_{j+2})   \Bigr] \Tr_{j+4}.
\end{align*}

\medskip
In particular, due to the complexity of the expressions involved, we point out that simply knowing the expression for the conservation laws for the NLS \eqref{equ:NLS_intro} is insufficient to determine the form of the operators $\W_n^j$.  In fact, the recursive formulation of these operators which we derive plays a crucial role in our proof. 

\medskip
We now record a paraphrased statement of our main result, Theorem \ref{main}, without providing details for the moment:

\begin{namedthm*}{Summary of Main Result}
Let  $(\gamma^{(k)}(t))_{k\in\bN}$ be a sufficiently regular solution of the (de)focusing cubic GP hierarchy \eqref{equ:gp_intro}, satisfying a certain admissibility condition (see Definition \ref{defn:admissible}). Then for each $n \in \bN$, $1 \leq j \leq n$ such that 
$k \geq j+n-1$ there exists an operator $\W_n^j$ such that the quantity
\[ 
\Tr \W_{n}^j  \gamma^{(k)}(t) 
\]
is conserved in time, namely
\[ 
\Tr \W_{n}^j  \gamma^{(k)}(t)  = \Tr \W_{n}^j  \gamma^{(k)}(0) .
\]
\end{namedthm*}

\begin{rmk} \label{rem:recover}
In addition to new integrals of motion, our main result also recovers the energy defined by Chen-Pavlovi\'{c}-Tzirakis \cite{CPT}, which corresponds to $W_3^j$ in the above notation. Moreover, the operators $\W_n^j$ also enable us to recover the higher order energies defined by Chen-Pavlovi\'{c} in \cite{CP14gwp}, see Remark \ref{rmk:gen_cons} and Proposition \ref{prop:higher_cons} for more details.
\end{rmk}

This result is, to the best of our knowledge, the first time an infinite sequence of independent conserved quantities for the cubic GP hierarchy has been produced, and the first time these operators $\W_n^j$ have been defined for the GP hierarchy. As is the case with the integrable structure for the cubic NLS, our main result applies to both focusing and defocusing hierarchies. 

Armed with the right formulation of the operators $\W_{n}^j $, we will prove that these conserved quantities can be given by averages of conserved quantities for the cubic NLS with respect to the so-called de Finetti measure of the initial data defined in \eqref{gkdf}. This should be compared to the scattering result of \cite{CHPS-sc} which states that the scattering states for the GP hierarchy are given by an appropriate average over the scattering states for the solutions to the NLS with respect to the de Finetti measure. Consequently, we believe that our definition yields conservation laws which are quite natural from a physical point of view, see Remark \ref{rmk:natural} for more details. We will also prove, in Corollary \ref{cor:op_form}, that these operators take the form of a leading $n$-th order differential operator plus lower order terms.

In particular, our result provides strong evidence of a link between an infinite particle system derived from many-body quantum statistical mechanics and the algebraic properties of the one-dimensional cubic NLS. There still remains the interesting open question of determining an integrable structure at the level of the GP hierarchy.

\subsection*{Organization of paper} In \S\ref{sec:gp_hierarchy} we briefly recall  the history of the derivation of the GP hierarchy from a many-body quantum system and some relevant facts about the GP hierarchy and the quantum de Finetti theorem. In \S\ref{sec:wn_def} we define the $\W_n^j$ operators described above. With the definition of the $\W_n$ operators in hand, we give the precise statement of our results in  \S\ref{sec:main}. In \S\ref{sec-w-fact} we prove the main results for factorized solutions and in \S\ref{sec:cons_laws_gp} we prove the results for general solutions of the GP hierarchy, and record some further consequences and applications of our formula for the conserved quantities.

\subsection*{Acknowledgements:} 
The authors thank the MSRI and the IHES for the kind hospitality that allowed us to develop this project. 
They would like to express their thanks to Patrick G\'erard for helpful discussions.

\section{On the Gross-Pitaevskii hierarchy} \label{sec:gp_hierarchy}


In this section we give a brief history of the derivation of the GP hierarchy, followed by a review of recent works on the GP hierarchy using the quantum de Finetti theorem.

\subsection{History of the derivation of the GP hierarchy from quantum many-particle systems} 

The derivation of the nonlinear Hartree (NLH) and the NLS equations
have been approached by many authors in a variety of ways. The first connection between quantum 
many particle systems and NLH was given by Hepp in \cite{Hepp}, and generalized by Ginibre and Velo \cite{GV1, GV2}, using the Fock space formalism and coherent states, which were inspired by techniques of quantum filed theory.  This method was further developed in \cite{RS, GMar, GMac, GMM1, GMM2, Chenx} and in the recent work \cite{brsc}. See also \cite{Pickl,Pickl2}. 

Another way to derive NLH and NLS is via the so called BBGKY\footnote[6]{Bogoliubov-Born-Green-Kirkwood-Yvon} hierarchy, which was prominently used in the works of
Lanford for the study of classical mechanical systems in the infinite particle limit \cite{Lan-1,Lan-2}.
The first derivation of the NLH via the BBGKY hierarchy was given
by Spohn in \cite{Spohn}. This topic was further studied in e.g. \cite{ABGT,AGT, frgrsc, FKS}. 
About a decade ago, Erd\"os, Schlein and Yau
fully developed the  BBGKY hierarchy approach to the derivation of 
the NLH and NLS in their seminal works \cite{ESY1,ESY2,ESY3,ESY4}, which initiated 
much of the current widespread interest in this research topic.

 The subject of the derivation of NLH and NLS from quantum many particle systems is closely related to the phenomenon of Bose-Einstein condensation, in systems of interacting bosons, 
which was first experimentally verified in 1995. For the mathematical study of BEC, we refer to 
\cite{ailisesoyn,LSSY} and references therein.

The most difficult step in the derivation of the GP hierarchy from many body quantum systems is the proof of uniqueness for solutions of the limiting hierarchy. The proof of uniqueness was first obtained by Erd\"os-Schlein-Yau \cite{ESY1,ESY2,ESY3,ESY4} for $d=3$ 
in the space $\{\gamma^{(k)} \, | \, \|\gamma^{(k)}\|_{\frh^1_k}  \, < \, \infty\}$, given in Definition \ref{equ:banach_space_density_matrices} below. Proving uniqueness is a crucial, and very involved part of Erd\"os-Schlein-Yau's approach for deriving the cubic defocusing NLS in ${\R}^3$ and is based on a powerful combinatorial method that resolves the issues that stem from the factorial growth of the number of terms in the  iterated Duhamel expansions.


The proof of uniqueness for solutions of the GP hierarchy was revisited by Klainerman-Machedon in \cite{KM08}, based on a reformulation of the combinatorial argument of Erd\"os-Schlein-Yau, and a viewpoint inspired by methods from nonlinear partial differential equations. This motivated many recent works, including \cite{KSS,CP10,CP11, Chenx, CH16, GSS}. In particular, the proof in \cite{KM08} provided strong indication that methods from dispersive PDE could be brought to bear upon the study of the GP hierarchy, and even on the quantum many-body system. Indeed, the techniques that were developed for the well-posedness theory of the GP were in turn successfully adapted in \cite{CP11} to the quantum many-particle systems.

Recently, a new methodology has entered the stage, based on a quantum de Finetti theorem which allows, in some instances, a more direct transfer of the techniques from the nonlinear Schr\"{o}dinger equation to the GP hierarchy \cite{CHPS}. We describe some of these ideas below in our context.

\subsection{The quantum de Finetti theorem and uniqueness of solutions to the GP hierarchy}  \label{sec:quantum_de_finetti}

\medskip
In the proof of our main result, Theorem \ref{main}, we use the version of uniqueness of solutions to the  GP hierarchy, which has been recently established in \cite{HTX} by employing a quantum de Finetti theorem, which is a quantum analogue of the Hewitt-Savage theorem in probability theory, \cite{HS}. Before we state the uniqueness result, we recall the strong version of the de Finetti's theorem,  due to Hudson-Moody \cite{HM} and Stormer \cite{Stormer}, and  stated in the context  of $C^*$-algebras. We quote this theorem using  the formulation presented by Lewin-Nam-Rougerie in  \cite{LNR}, which pioneered recent applications of the theorem in the context of questions related to Bose-Einstein condensation and derivations of nonlinear dispersive PDE from quantum many body systems. 
 
We start by recalling the definition of admissibility for a sequence of bosonic density matrices:  
 
\begin{defn} \label{defn:admissible}
Let $\cH$ be a separable Hilbert space and $\cH^k = \bigotimes_{sym}^k\cH$ the corresponding bosonic $k$-particle space. We say that 
\[
\Gamma = (\gamma^{(1)},\gamma^{(2)},\dots)
\]
is a sequence of admissible bosonic density matrices on  $\cH$ if $\gamma^{(k)}$ is a non-negative trace class operator on $\cH^k$ such that 
\[ 
\gamma^{(k)}=\tr_{k+1} \gamma^{(k+1)}.
\]
\end{defn}

\begin{thm}\label{thm-strongDeFinetti}
(Strong Quantum de Finetti Theorem)
Let $\cH$ be a separable Hilbert space and  $\cH^k = \bigotimes_{sym}^k\cH$ the corresponding bosonic $k$-particle space. If $\Gamma$ is a sequence of admissible  bosonic density matrices on  $\cH$, then there exists a unique Borel probability measure $\mu$, supported on the unit sphere $S\subset\cH$, and invariant under multiplication of $\phi \in \cH$ by complex numbers of modulus one, such that 
\begin{equation}\label{gkdf}
    \gamma^{(k)} = \int d\mu(\phi)  (  | \phi  \rangle \langle \phi |  )^{\otimes k}\;, \qquad \forall k\in\bN\,.
\end{equation} 
\end{thm}

The solutions to the GP hierarchy \eqref{equ:gp_intro} obtained via weak-* limits from solutions of the many body system as in \cite{ESY1}-\cite{ESY4} are not necessarily admissible. However a weak version of the quantum de Finetti theorem applies in those cases. Here we give the statement of the version of a weak quantum de Finetti theorem that was recently revived by Lewin-Nam-Rougerie \cite{LNR_hartree}. We note that an equivalent result had been established  also by Ammari-Nier in \cite{AN1,AN2}.

\begin{thm}\label{thm-weakDeFinetti} (Weak Quantum de Finetti Theorem) \cite{LNR_hartree, AN1,AN2}.
Assume that $\gamma_N^{(N)}$ is an arbitrary sequence of mixed states on $\cH^N$, $N\in\bN$, satisfying $\gamma_N^{(N)}\geq 0$ and $\tr_{\cH^N}(\gamma_N^{(N)})=1$, and assume that its $k$-particle marginals have weak-* limits 
\begin{equation}  
    \gamma^{(k)}_{N}:=\tr_{k+1,\cdots,N}(\gamma^{(N)}_N)
    \; \rightharpoonup^* \; \gamma^{(k)} \;\;\;\; (N\rightarrow\infty)\,,
\end{equation} 
in the trace class on $\cH^k$ for all $k\geq1$. Then, there exists a unique Borel probability measure $\mu$ on the unit ball in $\cH$, and invariant under multiplication of $\phi \in \cH$ by complex numbers of modulus one, 
such that
\begin{equation}\label{gkdf-w}
    \gamma^{(k)} = \int d\mu(\phi)  (  | \phi  \rangle \langle \phi |  )^{\otimes k}    \;, \qquad \forall k\in\bN\,.
\end{equation}  
\end{thm}

Motivated by the de Finetti theorems, which imply that a suitable member of an infinite hierarchy is on average factorized, Chen-Hainzl-Pavlovi\'{c}-Seiringer \cite{CHPS} obtained a new proof of the unconditional uniqueness of mild solutions to the cubic GP hierarchy in $\R^3$. Loosely speaking, this implies that such solutions to the GP hierarchy are given as an average of factorized solutions, where each factor is a solution to a cubic NLS. 
Following  \cite{CHPS}, de Finetti theorems were successfully used to address unconditional uniqueness of certain GP hierarchies, see e.g. \cite{Sohinger15,HTX,HTX2,CS14,HS16}. In particular, in the work at hand we use the unconditional uniqueness result for solutions to the cubic GP hierarchy in $\R^d, \, d\geq 1$, obtained recently in \cite{HTX}. Before we state this unconditional uniqueness result, we recall the definition of the space  $\frH^{\alpha}$ of sequences $(\gamma^{(k)})_{k\in\bN}$ and the definition  of a {\em mild solution} to \eqref{equ:gp_intro} in the space $L^\infty_{t\in[0,T]}\frH^\alpha$, as used in \cite{CHPS} and \cite{HTX}. 

\begin{defn}  \label{equ:banach_space_density_matrices}
For $\alpha > 0$, the space  $\frH^{\alpha}$ of sequences $(\gamma^{(k)})_{k\in\bN}$ is given by
 \begin{align}\label{frakH-def-1}
 	\frH^{\alpha}:=\Big\{ \,(\gamma^{(k)})_{k\in\bN} \, \Big| \, 
 	\|\gamma^{(k)}\|_{\frh^{\alpha}_k} 
	< R^{2k} \; \mbox{for some constant }R<\infty \, \Big\},
\end{align}
where 
\[
\|\gamma^{(k)}\|_{\frh^{\alpha}_k} \, := \,  \tr (|S^{(k,\alpha)}\gamma^{(k)}|)   \,,
\]
with
\begin{equation}\label{skalpha}
S^{(k,\alpha)}:=\prod_{j=1}^k(1-\Delta_{x_j})^{\alpha/2}(1-\Delta_{x_j'})^{\alpha/2}
\end{equation}
for $\alpha>0.$
\end{defn}

We state the  definition of a {\em mild solution} to the GP hierarchy \eqref{equ:gp_intro} in the space $L^\infty_{t\in[0,T]}\frH^\alpha$.
\begin{defn} 
Let 
\begin{equation} 
     U^{(k)}(t) := \prod_{\ell=1}^k e^{it(\Delta_{x_\ell}-\Delta_{x_\ell'})}
\end{equation} 
denote the free $k$-particle propagator. For $\alpha > 0$, a {\em mild solution} to \eqref{equ:gp_intro} in the space $L^\infty_{t\in[0,T]}\frH^\alpha$ 
is a sequence of marginal density matrices $\Gamma=(\gamma^{(k)}(t))_{k\in\bN}$ 
solving the integral equation
\begin{equation} 
    \gamma^{(k)}(t) = U^{(k)}(t)\gamma^{(k)}(0) + i \int_0^t  U^{(k)}(t-s) B_{k+1}\gamma^{(k+1)}(s) ds \,
    \;\;\;\;\;\; k\in\bN\,,
\end{equation} 
satisfying  
\begin{equation} 
    \sup_{t\in[0,T]}
    \|\gamma^{(k)}\|_{\frh^{\alpha}_k}  
    < R^{2k}
\end{equation} 
for a finite constant $R$ independent of $k$. 
\end{defn} 

Now we give the precise formulation of the unconditional uniqueness theorem we will use in our proofs. To the best of our knowledge, this is the only version of this theorem which applies in our setting, and yields a representation formula in dimension $d= 1$.
 
\begin{thm}\label{thm-uniqueness-2}  (Hong-Taliaferro-Xie) \cite{HTX}.
Let 
\begin{align}
	s \,  \Biggl\{
	\begin{array}{rcl}
	\geq &\frac{d}{6} & {\rm if} \; d=1,2 \\ 
	>&\frac{d-2}{2} & {\rm if} \; d\geq3 \,.
	\end{array} \Biggr.
\end{align}
If $(\gamma^{(k)}(t))_{k\in\bN}$ is a mild solution in $L^\infty_{t\in[0,T)}\frH^s$ 
to the (de)focusing cubic GP hierarchy in $\R^d$ with initial data $(\gamma^{(k)}(0))_{k\in\bN}\in\frH^s$, 
which is either admissible, or obtained
at each $t$ from a weak-* limit, 
then, $(\gamma^{(k)})_{k\in\bN}$ is the unique solution for the given initial data.

Moreover, if $(\gamma^{(k)}(0))_{k\in\bN} \in\frH^s$ satisfies
\begin{equation} 
    \gamma^{(k)}(0) = \int d\mu(\phi_0)(|\phi_0\rangle\langle\phi_0|)^{\otimes k} 
    \;\;\;\;\;\;\forall k\in\bN\,,
\end{equation} 
where $\mu$ is a Borel probability measure
supported  either  on the unit sphere or on the unit ball in $L^2(\R^d)$,
and invariant under multiplication of 
$\phi_0 \in L^2(\R^d)$ by complex numbers of modulus one, 
then  
\begin{equation}\label{eq-GPdeF-NLS-anyD2}
    \gamma^{(k)}(t) = \int d\mu(\phi_0)(|\phi(t,x)\rangle\langle \phi(t,x)|)^{\otimes k} 
    \;\;\;\;\;\;\forall k\in\bN\,,
\end{equation} 
where  for $t\in[0,T)$ $\phi(t, x)$ is the  solution 
to the cubic (de)focusing NLS with initial data $\phi_0:=\phi(0,x)$. 
\end{thm}

We end this subsection by proving the following result which we will crucially, and implicitly, use in many of our proofs. It essentially states that if the initial data $  \gamma^{(k)}(0) $ is smoother than $L^2$, then the de Finetti integration takes place on a smoother domain. More precisely we have the following proposition.
\begin{prop}\label{prop:bdd}
If   $(\gamma^{(k)}(0))_{k\in\bN} \in\frH^s$, satisfies the assumptions of the weak or strong quantum de Finetti theorem, then 
\begin{equation}\label{extrasmooth}
    \gamma^{(k)}(0) = \int_{C_s}  d\mu(\phi_0) (|\phi_0\rangle\langle\phi_0|)^{\otimes k} 
    \;\;\;\;\;\;\forall k\in\bN\,,
\end{equation}  
where $C_s$ is a bounded subset of $L^2\cap H^s$.
\end{prop} 

To prove this proposition, we begin by recording \cite[Lemma 4.5]{CHPS}, and its proof for completeness.

\begin{lem}[\protect{\cite[Lemma 4.5]{CHPS}}]\label{smeasure}
Let $s > 0$ and let $\mu$ be a Borel probability measure
supported  either on the unit sphere or on the unit ball in $L^2(\R^d)$.
If 
\begin{equation}\label{lem-mu-Hs-if} 
\int d\mu(\phi) \|\phi\|_{H^s({\R}^d)}^{2k} < M^{2k}, 
\end{equation} 
for some finite $M>0$ and all $k \in \bN$, 
then 
\begin{equation}\label{lem-mu-Hs-conc}
\mu \left( \{ \phi \in L^2({\R}^d) \; | \; \|\phi\|_{H^s({\R}^d)} > \lambda \} \right) = 0, 
\end{equation} 
for all $\lambda > M$. 
\end{lem} 

\begin{proof} 
By Chebyshev's inequality, for any $\lambda > 0$ and any $k \in \bN$
we have 
\begin{align}
\mu \left( \{ \phi \in L^2({\R}^d) \; | \; \|\phi\|_{H^s({\R}^d)} >\lambda \} \right)  &\leq \frac{1}{\lambda^{k}} \int d\mu(\phi) \|\phi\|_{H^s({\R}^d)}^{2k} \nonumber \\
&\leq \frac{M^{2k}}{\lambda^{2k}},
\end{align}
which for $\lambda > M$ goes to zero as $k \rightarrow \infty$. 
\end{proof} 

Now, we will apply this Lemma to prove Proposition \ref{prop:bdd}:

\begin{proof}[Proof of Proposition \ref{prop:bdd}]
This proof again uses an argument analogous  to the one  presented in \cite[Lemma 4.5]{CHPS}. Since $(\gamma^{(k)}(0))_{k\in\bN} \in\frH^s$, one has that 
\begin{equation}\label{1extrasmooth}
Tr(|S^{(k,s)}[\gamma^{(k)}(0)]|)\lesssim M^{2k} \,\,, \qquad \forall k\in\bN\,,
\end{equation}
for some $M>0$. We note that \eqref{1extrasmooth} is equivalent to
\begin{equation}\label{extraphi}
\int  d\mu(\phi_0) \|\phi_0\|_{H^s}^{2k}\lesssim M^{2k},
\end{equation}
thus our assumptions yield \eqref{lem-mu-Hs-if}, and by \eqref{extraphi} and Lemma \ref{smeasure}, \eqref{extrasmooth} follows.
\end{proof}

\begin{rmk}
Often we will suppress the time dependence of the solutions $\phi$ to the nonlinear Schr\"odinger equation, writing $\phi(t,x) = \phi(x)$, but we will make the dependence explicit where necessary.
\end{rmk}

\section{Definition of the $\W_{n}^j $  operators on density matrices} \label{sec:wn_def}

In this section, we explicitly define the operators  $\{\W_{n}^j \}_{n}$ which appear in the summary statement of our main result. In \S\ref{sub-W-f}, we will first define the operators on factorized density matrices and show that they are well-defined, bounded linear operators. Then in \S\ref{functional} we present some functional analytic tools that we shall use in \S\ref{sub-W} to extend the definition of  the operators  $\{\W_{n}^j \}_{n}$ from factorized to arbitrary density matrices.  


\subsection{The $\W_{n}^j $ operators on factorized density matrices} \label{sub-W-f} 

We start by introducing the operators   $\{\W_{n}^j \}_{n}$ that act on factorized density matrices. We will use the function spaces $\mathfrak{h}_{m}^{ \alpha} $ from Definition \ref{equ:banach_space_density_matrices}. We note that the appearance of the index $j$ in the following definition is necessary to produce a consistent recursive definition, but in our application we will fix $j=1$.

\begin{defn} \label{def:GP-f-w}
For a non-negative integer $n$ and any integer $1\leq j \leq n+1$ and $\alpha > 0$ we define the operators\footnote[7]{Here, we abuse notation and denote the space of factorized solutions using the same notation as for general solutions. As we will see, these $\W_n^j$ extend to the whole space, see Proposition \ref{prop:ops_defined}.}
\[
\W^j_{n+1}: \mathfrak{h}_{n+1}^{n/2 + \alpha} \rightarrow \mathfrak{h}_{1}^{ \alpha}  
\]
on factorized density matrices via:  

\begin{align} 
&\W_1^1  \Biggl( \phi(x_1) \barphi(x'_1) \Biggr) 
 =   \phi(x_1) \barphi(x'_1)  \label{stat-equ:GP-f-w1}\\
&\W_{n+1}^j \Biggl(  \prod_{l=j}^{j+n} \phi(x_\ell) \barphi(x'_\ell) \Biggr) 
 = -i \partial_{x_j} \W_{n}^{j} \Tr_{j+n} \Biggl(  \prod_{l=1}^{j+n} \phi(x_\ell) \barphi(x'_\ell) \Biggr) \nonumber \\
& + \; \; \; \; \sum_{k=1}^{n-1} B_{j,j+k} 
{\color{red}{
\W_{k}^{j} 
}}
\otimes 
{\color{blue}{
\W_{n-k}^{j+k} 
}}
\Tr_{j+n}
\Biggl(  
{\color{red}{
\prod_{l=j}^{j+k-1} \phi(x_\ell) \barphi(x'_\ell)   
}}
{\color{blue}{
\prod_{l=j+k}^{j+n-1} \phi(x_\ell) \barphi(x'_\ell) 
}}
\Biggr)  
\phi(x_{j+n}) \barphi(x'_{j+n}), 
\label{stat-equ:GP-f-wn} 
\end{align}
for any $\phi \in H^{n/2 +\alpha}$ and such that $\|\phi\|_{L^2} = 1$. 
\end{defn} 

\medskip
%
The operator
$\W_{k}^{j} \otimes \W_{n-k}^{j+k} \Tr_{j+n}$ acts on   $\Biggl(  \prod_{l=j}^{j+n} \phi(x_\ell) \barphi(x'_\ell) \Biggr)$ 
in the following way: 
\begin{itemize} 
\item   By recalling the definition of a partial trace and keeping in mind the normalization  $\|\phi\|_{L^2} = 1$, we note that 
$\Tr_{j+n}$ acts on   $\Biggl(  \prod_{l=j}^{j+n} \phi(x_\ell) \barphi(x'_\ell) \Biggr)$ 
so that we have: 
\[
	\Tr_{j+n} \Biggl(  \prod_{l=j}^{j+n} \phi(x_\ell) \barphi(x'_\ell) \Biggr)
	= \int \; dx_{j+n}  \; dx'_{j+n}  \; \delta(x_{j+n} - x'_{j+n}) \; \prod_{l=j}^{j+n} \phi(x_\ell) \barphi(x'_\ell) 
	=   \prod_{l=j}^{j+n-1} \phi(x_\ell) \barphi(x'_\ell) 
\]
\item $\W_{k}^{j}$ acts on  $k$ double-factors 
\[
\prod_{l=j}^{j+k-1} \phi(x_\ell) \barphi(x'_\ell) =  \phi(x_j) \barphi(x'_j) \cdot \dots \cdot  \phi(x_{j+k-1}) \barphi(x'_{j+k-1})
\]
\item $\W_{n-k}^{j+k}$ acts on the next $n-k$ double-factors 
\[
\prod_{l=j+k}^{j+n-1} \phi(x_\ell) \barphi(x'_\ell) =  \phi(x_{j+k}) \barphi(x'_{j+k}) \cdot \dots \cdot  \phi(x_{j+k+n-k-1}) \barphi(x'_{j+k+n-k-1}).
\]
\end{itemize}

In order to show that the operators $W_{n}^j$ in Definition \ref{def:GP-f-w} are well-defined bounded linear operators, we first point out the following boundedness property of the 
collision operator $B_{i,j}$: 

\begin{lem}\label{lem:b_bounds}
The collision operator $B_{i,j}: \mathfrak{h}_k^1 \to \mathfrak{h}_{k-1}^0 $ is a bounded linear operator on density matrices.
\end{lem}
\begin{proof}
We will prove this statement for tensorized functions, and the general statement follows from density. The linearity of $B_{i,j}$ is immediate from its definition as a convolution operator. First we consider $k=2$ and we compute
\[
B_{i,j} f(x_i) \overline{f}(x_i') g(x_j) \overline{g}(x_j') = f(x_i) \overline{f}(x_i') |g(x_i)|^2.
\]
We note that for factorized density matrices $\gamma^{(k)}(\ux_k, \ux_k') = \prod_{i=1}^k \phi(x_i) \overline{\phi}(x_i')$, one has $\gamma^{(k)}(\ux_k, \ux_k') \in \mathfrak{h}_k^1$ if and only if each one-particle function $\phi(x) \in H^1$. Now,
\[
\| B_{i,j} f(x_i) \overline{f}(x_i') g(x_j) \overline{g}(x_j') \|_{ \mathfrak{h}_{1}^0} = \textup{Tr}(|B_{i,j} f(x_i) \overline{f}(x_i') g(x_j) \overline{g}(x_j')| )
\]
and we compute
\[
\textup{Tr}(|B_{i,j} f(x_i) \overline{f}(x_i') g(x_j) \overline{g}(x_j')| )  = \int |f(x)|^2 |g(x)|^2 \leq \Biggl( \int |f(x)|^4 \Biggr)^{1/2} \Biggl( \int |g(x)|^4 \Biggr)^{1/2} 
\]
using H\"older's inequality. So the right-hand side is bounded for $f, g \in H^1$ by the Gagliardo-Niremberg inequality, and hence we have
\[
\textup{Tr}(|B_{i,j} f(x_i) \overline{f}(x_i') g(x_j) \overline{g}(x_j')| )  \leq \|f\|_{H^1}^{2} \|g\|_{H^1}^{2} \leq \| f \otimes g \|_{\mathfrak{h}_2^1}.
\]
 For $k > 2$, we may extend the above argument by tensoring the operator $B_{i,j}$ with the identity.
\end{proof}

We are now prepared to justify our definition of the operators $\W_{n}^j$ in Definition \ref{def:GP-f-w}.

\begin{prop}\label{prop:ops_bddness_defn}
Let $n \in \bN$ and $1 \leq j \leq n+1$ and $\alpha \geq 0$. The operators $\W_{n}^j$ in Definition \ref{def:GP-f-w} are well-defined bounded linear operators
\[
\W_{n}^j : \mathfrak{h}_{n}^{(n-1)/2 + \alpha} \to \mathfrak{h}_{1}^{\alpha}.
\]
\end{prop}

\begin{proof}
We prove this claim by complete induction. The case $n=1$ is clear (with $j \leq 2$), and the statement about linearity follows  since the composition of linear operators is linear, so we focus on proving the boundedness claim. Suppose that the assumptions hold for all $1 \leq i \leq n$ and $1 \leq j \leq i+1$. Then
\[
W^j_{n+1} = - i  \partial_{x_j} \W_{n}^{j} \Tr_{j+n} + \sum_{k=1}^{n-1} B_{j,j+k} 
{\color{red}{
\W_{k}^{j} 
}}
\otimes 
{\color{blue}{
\W_{n-k}^{j+k} 
}}
\Tr_{j+n}.
\]
We index the variables as in \eqref{stat-equ:GP-f-wn} and consider the terms separately. The first term is a composition of three operators, which we can write as
\[
\mathfrak{h}_{n+1}^{n/2 + \alpha} \xrightarrow{\Tr_{j+n}} \mathfrak{h}_{n}^{n/2 + \alpha} \xrightarrow{\W_{n}^{j}} \mathfrak{h}_{1}^{1/2 + \alpha} \xrightarrow{\partial_{x_j}} \mathfrak{h}_{1}^{\alpha},
\]
where the properties of the operator $\W_n^j$ in the middle of this composition come from the inductive hypothesis. We now examine the second term. Here, we break the sum into components and we analyze each one. We first observe that
\[
\mathfrak{h}_{n+1}^{n/2 + \alpha} \xrightarrow{\Tr_{j+n}} \mathfrak{h}_{n}^{n/2 + \alpha} \simeq \mathfrak{h}_{k}^{n/2 + \alpha} \otimes \mathfrak{h}_{n-k }^{n/2 + \alpha} .
\]
Next, for $n \geq 2$ and $j$ as above, we have $n-k, k \geq 1$ in the sum, and hence we can decompose these spaces
\[
\mathfrak{h}_{k}^{n/2 + \alpha} \otimes \mathfrak{h}_{n-k}^{n/2 + \alpha} \simeq \mathfrak{h}_{k}^{(k-1)/2 + (n-k + 1)/2 + \alpha} \otimes \mathfrak{h}_{n-k}^{(n-k-1)/2 + (k+1)/2 + \alpha} \subseteq \mathfrak{h}_{k}^{(k-1)/2 + 1 + \alpha} \otimes \mathfrak{h}_{n-k}^{(n-k-1)/2 + 1 + \alpha},
\]
where the inclusion comes from the fact that we are working in inhomogeneous spaces.  Next, we write
\[
\mathfrak{h}_{k}^{(k-1)/2 + 1 + \alpha} \otimes \mathfrak{h}_{n-k }^{(n-k-1)/2 + 1 + \alpha} \xrightarrow{\W_{k}^{j} \otimes \W_{n-k}^{j+k}  } \mathfrak{h}_{1}^{1 + \alpha} \otimes \mathfrak{h}_{1}^{1 + \alpha} \simeq \mathfrak{h}_{2}^{1 + \alpha} \xrightarrow{B_{j, j+k}} \mathfrak{h}_{1}^{\alpha},
\]
where in the first step we use the inductive hypothesis, and in the second step we use  Lemma \ref{lem:b_bounds}. Thus, every term in the sum maps
\[
\mathfrak{h}_{n+1}^{n/2 + \alpha} \to \mathfrak{h}_{1}^{\alpha}
\]
and consequently a finite sum of such operators does as well. This proves the inductive step and completes the proof.
\end{proof}

\begin{rmk}\label{identity_tensor}
We note that by tensoring with the identity, we may extend the operator $\W_{n}^j $ to tensor products of arbitrary lengths $m \geq n$, that is we define the extension $\widetilde{\W}_{n}^j $ as
\[
\widetilde{\W}_{n}^j : \mathfrak{h}_{m}^{(n-1)/2 + \alpha} \to \mathfrak{h}_{m-n+1}^{\alpha}, \qquad \widetilde{\W}_{n}^j = \underbrace{I \otimes \cdots \otimes I}_{\ell\,\, times} \otimes \W_{n}^j \otimes \underbrace{I \otimes \cdots \otimes I}_{m-n+1 -\ell\,\, times} 
\]
for $0 \leq \ell \leq m - n +1$. In the sequel, we will denote such operators using the same notation $\W_{n}^j $ and the domain of the operator will be made explicit from the context.
\end{rmk}

The next step is to extend our definition of the operators $\W_n^j$ in \eqref{def:GP-f-w}-\eqref{stat-equ:GP-f-wn}  to an admissible sequence of density matrices. We will achieve this via using some functional analytic tools that we gather in the following subsection.

\subsection{Functional analytic toolbox} \label{functional}
We record here some basics on Bochner integrals. 

In the sequel, let $(A, \cA, \mu)$ be a $\sigma$-finite measure space and $E$ a Banach space. For measurable subsets $A_n$ we say $f$ is a $\mu$-simple function if
\[
f = \sum_{i=1}^n x_n \chi_{A_n}
\]
for $x_n \in E$, $\mu(A_n) < \infty$, and $\chi_{A_n}$ the characteristic function of $A_n$.
\begin{defn}
A function $f : A \to E$ is said to be $\mu$-Bochner measurable if there exists a sequence $\{f_n\}_{n \in \bN}$ of $\mu$-simple functions converging to $f$ $\mu$-almost everywhere.
\end{defn}

\begin{defn}
A function $f : A \to E$ is said to be $\mu$-Bochner integrable if there exists a sequence of $\mu$-simple functions $f_n : A \to E$ such that
\begin{enumerate}
\item $\lim_{n \to \infty} f_n = f$ $\mu$-almost everywhere,
\item $\lim_{n \to \infty} \int_A \|f_n - f\| d\mu = 0$,
\end{enumerate}
where the latter is the Lebesgue integral.
\end{defn}

We will rely on the following theorem of Hille, see for instance \cite[\S3.5]{HP}.
\begin{thm}[Hille] \label{thm:Hille}
Let $f : A \to E$ be $\mu$-Bochner integrable and let $T$ be a closed linear operator with domain $\cD(T)$ in $E$ taking values in a Banach space $F$. Assume that $f$ takes values in $\cD(T)$ $\mu$-almost everywhere and that the $\mu$-almost everywhere defined function $Tf: A \to E$ is $\mu$-Bochner integrable. Then $\int_A f d\mu \in \cD(T)$ and
\[
T \int_A f d\mu = \int_A Tf d\mu.
\]
\end{thm}

In our case, we take $f(\phi) = (|\phi \rangle \langle\phi|)^{\otimes (n+1)}$ which (by construction) takes values in the domain of the operator $\W_{n}^j$. In verifying the hypotheses of Hille's theorem in our work, the following classical result, due to Bochner, provides a convenient criterion for determining whether a given $\mu$-Bochner measurable function is $\mu$-Bochner integrable. 
\begin{thm}\label{thm:l1_boch_int}
Let $(A, \cA, \mu)$ be a $\sigma$-finite measure space, $E$ a Banach space and let $f: A \to E$ be $\mu$-Bochner measurable. Then $f$ is $\mu$ Bochner integrable if and only if 
\begin{align} \label{equ:l1_bds}
\int_A d\mu \|f\|_{E} < \infty,
\end{align}
in which case
\[
\Biggl\| \int_A d\mu f \, \Biggr\|_E \leq \int_A d\mu \|f\|_{E}. 
\]
\end{thm}
\begin{proof}
First suppose that $f$ is $\mu$-Bochner integrable, then the proof of \eqref{equ:l1_bds} follows from approximating $f$ by $\mu$-simple functions. The converse follows by letting $\{g_n\}_{n \in \bN}$ be a sequence of $\mu$ simple functions converging to $f$ and defining
\[
f_n := \mathbf{1}_{\{\|g_n\| \leq 2 \|f\|\}} g_n,
\]
and using the dominated convergence theorem for Lebesgue integrals.
\end{proof}

Finally, since we will use this result in the subsequent sections, we record the following functional analytic property of the operators defined in \S\ref{sub-W-f}, which follows directly from Proposition \ref{prop:ops_bddness_defn}.

\begin{cor}\label{cor:Wn_closed}
The linear operators $\W_n^j : \mathcal{D}(\W_n^j)  \to \mathfrak{h}_{1}^{0}$ are closed operators.
\end{cor}

\subsection{Extending the  $\W_n^j$ operators}  \label{sub-W}

Now we are ready to  extend our definition of the operators $\W_n^j$ in \eqref{def:GP-f-w}-\eqref{stat-equ:GP-f-wn}  to an admissible sequence of density matrices. We state the following proposition in more generality than we will need.

\begin{prop} \label{prop:ops_defined}
Let $n \in \bN$, $1 \leq j \leq n$ such that $N \geq j+n-1$. Let $\alpha \geq 0$, and let $(\nu, L^2(\R), \cB)$ be a $\sigma$-finite measure space with $\cB$ the Borel $\sigma$-algebra on $L^2(\R)$. Let $\gamma^{(n)} \in \mathfrak{h}_{n}^{(n-1)/2 }$ be any density matrix given by
\begin{align} \label{equ:density_meas}
\gamma^{(n)}(x,t) = \int d\nu(\phi)  \bigl(|\phi \rangle \langle \phi |\bigr)^{\otimes n},
\end{align}
where $\nu$ is a Borel measure with bounded support in $H^{{(n-1)/2 }}(\R)$ and such that its support is also contained in the unit ball or sphere in $L^2(\R)$. Then the operator $\W_n^j$ is a bounded linear operator $\W_n^j : \mathfrak{h}_{n}^{(n-1)/2 } \to \mathfrak{h}_{1}^0$ which satisfies
\begin{align}\label{equ:w_inside}
\W_n^j \gamma^{(n)}(x,t) = \int d\nu(\phi) \W_n^{j} (|\phi \rangle \langle \phi |)^{\otimes n}.
\end{align}
\end{prop}

\begin{proof}
We will prove this proposition using Hille's Theorem, stated in Theorem \ref{thm:Hille}, for the $\sigma$-finite measure space $(\nu, \cS, \cB)$, where $\cS \subset L^2(\R)$ is the unit ball or sphere.

For any $s \in \R$, and $k \in \bN$ we let 
\[
f_k: H^s \to \mathfrak{h}_k^s, \qquad \phi \longmapsto \prod_{i=1}^k \phi(x_i) \overline{\phi}(x_i').
\]
First we note that $f_k$ is continuous. Indeed, for $\phi, \psi \in H^s$,
\begin{align*}
&\Bigl\| \prod_{i=1}^k \phi(x_i) \overline{\phi}(x_i') - \prod_{i=1}^k \psi(x_i) \overline{\psi}(x_i')\Bigr\|_{\mathfrak{h}_k^s} \\
&= \Tr \biggl| \prod_{i=1}^k (1- \Delta_{x_i})^{s/2} (\phi(x_i) - \psi(x_i))(1- \Delta_{x_i'})^{s/2} \overline{\phi}(x_i') - (1- \Delta_{x_i'})^{s/2} (\overline{\phi}(x_i') - \overline{\psi}(x_i') (1- \Delta_{x_i})^{s/2} \psi(x_i))\biggr| \\
&\lesssim \|\phi - \psi\|^k_{H^s} \Bigl( \|\phi\|_{H^s} + \|\psi\|_{H^s} \Bigr)^k.
\end{align*}
Now fix $k = n$ and $s ={(n-1)/2 }$. We make two observations:
\begin{itemize}
\item The function $f_{n}$ is $\nu$-Bochner measurable, since it is continuous $H^{{(n-1)/2 }} \to \mathfrak{h}_{n}^{{(n-1)/2 }} \subseteq \mathfrak{h}_{n}^{0} $ and $\cB$ is the Borel $\sigma$-algebra on $L^2$. 
\item For any probability measure $\nu$ on the unit sphere in $L^2(\R)$ with bounded support in $H^{{(n-1)/2 }} $ we have
\begin{align*} \label{equ:l1_bds_f}
\int d\nu(\phi) \|f_{n}(\phi)\|_{ \mathfrak{h}_{n}^s} =  \int d\nu(\phi) \|\phi\|_{H^{{(n-1)/2}}}^{n} \leq C.
\end{align*}
where we have used Proposition \ref{prop:bdd} in the last inequality.
\end{itemize}
Consequently, we conclude from Theorem \ref{thm:l1_boch_int} that $f_{n}$ is $\nu$-Bochner integrable. By Corollary \ref{cor:Wn_closed}, we satisfy the closedness hypothesis of Hille's theorem. Furthermore, the function 
\[
f_{n}: H^{(n-1)/2} \to \mathfrak{h}_n^{(n-1)/2} \subseteq \mathfrak{h}_{n}^{0}
\]
 is continuous, hence takes values in $\mathfrak{h}_{n}^{0}$. Now, by Proposition \ref{prop:ops_bddness_defn} and the properties of the functions $f_n$ introduced above, we have
\[
\int d\nu(\phi) \| \W_n^j f_{n}(\phi) \|_{\mathfrak{h}_1^0} \leq C(n) \int d\nu(\phi) \|f_n(\phi)\|_{\mathfrak{h}_n^{(n-1)/2} } \leq C(n) \int d\nu(\phi) \|\phi\|^{n}_{H^{(n-1)/2} } 
\]
and we obtain the desired bound by Proposition \ref{prop:bdd}. Hence $\W_n^1 f_{n}(\phi)$ is $\nu$-Bochner integrable. Thus by Hille's theorem, stated in Theorem \ref{thm:Hille}, we have
\[
\int d\nu(\phi) f_n(\phi) \in \cD(\W_n^j),
\]
and we obtain 
\[
\W_n^j \int d\nu(\phi) f_n(\phi) = \int d\nu(\phi) \W_n^j f_n(\phi)
\]
as desired.
\end{proof}

\begin{rmk} \label{rmk:gen_defn}
Note that the only requirement in the previous proposition was that a given density matrix could be written as in \eqref{equ:density_meas}. In particular, this proposition will apply to any sequence of density matrices which is admissible, or obtained at each time $t$ from a weak-* limit. In those cases the quantum de Finetti theorem yields a $\sigma$-finite measure $\mu$ (which plays the role of the measure $\nu$ in the previous proposition) such that
\[
\gamma^{(k)} = \int d\mu(\phi) \bigl(|\phi\rangle \langle \phi| \bigr)^{\otimes k}
\]
for all $k \in \bN$. In particular, we have constructed a sequence of operators $\{\W_n^j\}$ that, for any such sequence of density matrices, and $k \geq n$ acts by
\begin{align} \label{equ:ops_sequence}
\W_n^j\gamma^{(k)} = \int d\mu(\phi) \W_n^j \bigl(|\phi\rangle \langle \phi| \bigr)^{\otimes k}.
\end{align}
Moreover, Proposition \ref{prop:ops_defined}, and consequently \eqref{equ:ops_sequence}, will hold at any fixed time for any solution to the GP hierarchy $\{\gamma^{(k)}(t)\}$ which arises from initial data satisfying the hypotheses of the quantum de Finetti theorem.
\end{rmk}

\section{Statement of the main results} \label{sec:main}

Now we are ready to precisely state the main results of this paper, specifically Theorem \ref{thm:GP-f-cons-1} and Theorem \ref{main}. In the sequel, analogously to the case of Sobolev space, we define, for convenience, the following topological space
\[
\frH^{\infty} = \bigcap_{\alpha = 0}^\infty \frH^{\alpha},
\]
where $\frH^{\alpha}$ was introduced in \eqref{frakH-def-1}.

Before proceeding, we briefly recall the conserved quantities for the NLS equation \eqref{equ:NLS_intro}. We follow the presentation in Faddeev-Takhtajan \cite{FT} which is based on the inverse scattering method of Zakharov-Shabat \cite{ZS}. We begin with the following definition of the sequence of functions which will be used to construct the conserved integrals  for the NLS.

\begin{defn} \label{def:NLS-w} 
Let $\phi \in C^{\infty}$ be a solution to the cubic NLS \eqref{equ:NLS_intro}. For a non-negative integer $n$ we define the functions $w_{n+1}$ by:  

\begin{align} 
w_1(x) & = \phi(x)  \label{equ:NLS-w1}\\
w_2(x) & = -i \partial_x \phi(x) \\
w_{n+1} (x) 
& = -i \partial_{x} w_{n}(x) + \kappa \barphi(x) \sum_{k=1}^{n-1} w_{k}(x) w_{n-k}(x), \qquad n \geq 2 \label{equ:NLS-wn}
\end{align}
\end{defn} 

\noindent The conserved integrals  for the NLS \eqref{equ:NLS_intro} can be expressed as in the following proposition, (see e.g. \cite[Chapter I, \S4 - 5]{FT}). 
\begin{prop} \label{prop:NLS-cons}
Let $\phi \in C^{\infty}$ be a solution to the cubic NLS \eqref{equ:NLS_intro}.  
Then for each $n \in \bN$, the integral 
\begin{equation} \label{equ:NLS-cons} 
 I_n(\phi) :=\int w_n(x) \barphi(x) \; dx 
 \end{equation} 
is conserved in time. 
\end{prop}

%

The crucial property of the operators  $\{\W_{n}^j \}_{n}$ defined in \eqref{stat-equ:GP-f-w1}--\eqref{stat-equ:GP-f-wn} , whose proof we postpone until \S\ref{sec-w-fact}, 
expresses the link between these operators and  the functions $\{w_{n}\}_{n}$, and can be written as follows: 
\begin{equation} \label{intro-equ:GP-f-Ww} 
	\W_{n}^j \Biggl(  \prod_{l=j}^{j+n-1} \phi(t,x_\ell) \barphi(t,x'_\ell) \Biggr) 
	 =  w_n(x_j) \barphi(t,x'_j),
 \end{equation} 
for every $n = 0, 1, 2, ...$, $1 \leq j \leq n+1$ and $\phi(t,x)$ a solution to the cubic (de)focusing NLS in a certain sufficiently regular Sobolev space.   
In particular, 
\begin{equation} \label{intro-equ:GP-f-Ww-1} 
	\W_{n}^1 \Biggl(  \prod_{l=1}^{n} \phi(t,x_\ell) \barphi(t,x'_\ell) \Biggr) 
	 =  w_n(x_1) \barphi(t,x'_1), 
\end{equation} 
which together with Proposition \ref{prop:NLS-cons} will imply the following result:

\begin{thm} \label{thm:GP-f-cons-1}
For  $\phi(t,x) \in L^{\infty}_{t \in [0,\infty)}C^{\infty}(\mathbb R)$
a solution to the cubic (de)focusing NLS, 
let 
\[
\biggl\{ \; \prod_{l=j}^{j+n-1} \phi(t,x_\ell) \barphi(t,x'_\ell)  \biggr\}_{n=1}^\infty
\]
be a sequence of factorized solutions  to the GP hierarchy in  $L^{\infty}_{t \in [0,\infty)} \mathfrak{H}^{\infty}$. 
Then for each $n \in \bN$, the quantity
\begin{equation} \label{equ:GP-f-cons0} 
 \Tr \W_{n}^1 \Biggl(  \prod_{l=1}^{n} \phi(t,x_\ell) \barphi(t,x'_\ell) \Biggr) 
\end{equation} 
is conserved in time, that is for any $t \in \R$,
\[
 \Tr \W_{n}^1 \Biggl(  \prod_{l=1}^{n} \phi(t,x_\ell) \barphi(t,x'_\ell) \Biggr)  =  \Tr \W_{n}^1 \Biggl(  \prod_{l=1}^{n} \phi(0,x_\ell) \barphi(0,x'_\ell) \Biggr) .
\]
\end{thm} 

We will prove this theorem by explicitly identifying the quantity in \eqref{equ:GP-f-cons0}  as the $n$-th conserved quantity for the NLS, $I_n(\phi)$ in \eqref{equ:NLS-cons}. We prove these claims, and consequently Theorem \ref{thm:GP-f-cons-1} in Proposition \ref{prop:GP-f-cons}.

\medskip
Now we are ready to state our main theorem.

\begin{thm}\label{main} 
If $(\gamma^{(k)}(t))_{k\in\bN}$ is a mild solution in $L^\infty_{t\in[0,\infty)}\frH^{\infty}$ to the (de)focusing cubic GP hierarchy in $\R$ which is admissible, then for each $n \in \bN$, $1 \leq j \leq n$ and $k \geq j+n-1$ the quantity 
\begin{equation} \label{equ:GP-f-cons2} 
 \Tr \W_{n}^j \gamma^{(k)}(t)
\end{equation} 
is conserved in time, that is, for any $t \in \R$,
\[
 \Tr \W_{n}^j \gamma^{(k)}(t) =  \Tr \W_{n}^j \gamma^{(k)}(0).
\]
Moreover, given the de Finetti measure $\mu$ for the sequence $(\gamma^{(k)}(0))_{k \in \bN}$ we obtain that 
\begin{align}\label{equ:cons_avg}
 \Tr \W_{n}^j \gamma^{(k)}(t) = \int d\mu\, I_n(\phi)
\end{align}
where $I_n(\phi)$ is the conservation law for the NLS given by \eqref{equ:NLS-cons}.
\end{thm} 

\noindent We will prove this theorem in \S\ref{sec:proof_of_main} after establishing Theorem \ref{thm:GP-f-cons-1} in \S\ref{sec-w-fact}.

\begin{rmk}
We note that if $\{ \gamma^{(k)} \}_{k =1}^\infty$ is any sufficiently regular admissible sequence of density matrices, then for each $n \in \bN$, $1 \leq j \leq n$ and any $K \geq k \geq j + n-1$, we have
\[
\Tr \W_{n}^j \gamma^{(k)} =  \Tr \W_{n}^j \gamma^{(K)},
\]
where as usual, we understand the operator $\W_{n}^j$ applied to $\gamma^{(K)}$ as an extension of the usual operator $\W_{n}^j$ by tensoring with the identity as in Remark \ref{identity_tensor}.
\end{rmk}

\begin{rmk}\label{rmk:natural}
As mentioned in the introduction, we view our main result in the spirit of \cite{CHPS-sc} which demonstrates that solutions of the GP hierarchy which scatter are suitable averages of the scattering solutions for the NLS. Indeed, Theorem \ref{main}, and in particular \eqref{equ:cons_avg}, demonstrates that conserved quantities for the cubic GP hierarchy can be obtained as an appropriate average of the conserved quantities for the NLS.
\end{rmk}

\begin{rmk} \label{rmk:gen_cons}
The operator $\W_3^j$ corresponds to the conserved quantity of Chen, the third author and Tzirakis \cite{CPT}, thus we recover the conserved quantity which corresponds to the Hamiltonian for the NLS.
We also note that operators of the form $\W_3^1 \otimes \W_3^4 \otimes \W_3^7 \otimes \ldots \otimes \W_3^{1 + k\cdot 3}$ recover the higher conservation laws of \cite{CP14}. Indeed, one has for instance that
\begin{align} \label{equ:higher_cons_law}
\Tr \left((\W_3^1 \otimes \W_3^4 ) \prod _{i=1}^6 \phi(x_i) \phi(x_i')\right) = \Bigl( I_1(\phi) \Bigr)^2,
\end{align}
and the conservation for the equivalent expression  for general density matrices follows from the arguments we will use to prove Theorem \ref{main}. We will return to this point in \S\ref{sec:cons_laws_gp}. Finally, we observe that any higher order operator with the form
\[
\W_{n_1}^1 \otimes \W_{n_2}^{1 + n_1} \otimes \W_{n_3}^{1 + n_1 + n_3}  \otimes \ldots  \otimes \W_{n_k}^{1 + n_1+ \ldots+ n_{k-1}}
\] 
will also generate conservation laws for the GP hierarchy, although as in \eqref{equ:higher_cons_law} such operators will yield polynomials in the conserved quantities associated to the indices $n_i$. See Proposition \ref{prop:higher_cons} for more details.
\end{rmk}


\section{Conservation laws for factorized solutions of the GP hierarchy} \label{sec-w-fact}

In this section we prove  Theorem \ref{thm:GP-f-cons-1} which identifies infinitely many conserved quantities for 
factorized solutions to the GP hierarchy \eqref{equ:gp_intro}.  We start by specifying the crucial link between the operators $\{\W_{n}^j \}_{n}$ defined in \eqref{stat-equ:GP-f-w1}--\eqref{stat-equ:GP-f-wn}  and the functions $\{w_{n} \}_{n}$ from the conserved quantities for the NLS. We establish this link in the following proposition.

\begin{prop} \label{prop:GP-f-Ww} 
Suppose $\phi(t,x) \in L^{\infty}_{t \in [0,\infty)}C^{\infty}(\mathbb R)$ is a solution to the cubic (de)focusing NLS \eqref{equ:NLS_intro}.  
Then the operators $\{\W_{n}^j \}_{n}$ satisfy the following identity: 
\begin{equation} \label{equ:GP-f-Ww} 
\W_{n}^j \Biggl(  \prod_{l=j}^{j+n-1} \phi(t,x_\ell) \barphi(t,x'_\ell) \Biggr) 
 =  w_n(x_j) \barphi(t,x'_j),
 \end{equation} 
for every $n \in \bN$ and $1 \leq j \leq n+1$. 
\end{prop} 

\begin{proof} 
We prove the claim via mathematical induction in $n$.
First, we note that in the case when $n=0$ the identity \eqref{equ:GP-f-Ww} is satisfied thanks to the definition \eqref{stat-equ:GP-f-w1}
of  $\W^1_1$ and the definition  \eqref{equ:NLS-w1} of $w_1$, which immediately give that
\begin{equation} \label{equ:GP-f-Ww1}
\W_{1}^1 \Biggl( \phi(t,x_1) \barphi(t,x'_1) \Biggr) 
 =  \phi(t,x_1) \barphi(t,x'_1) 
=  w_1(x_1) \barphi(t,x'_1). 
\end{equation}

Now suppose that \eqref{equ:GP-f-Ww} is satisfied for all integers $n$, $1 \leq n \leq m$ and all integers $j$ such that $1 \leq j \leq n+1$. Then we prove the identity for $n=m+1$ as follows. 
Fix an integer $1 \leq j \leq m+2$. By the definition  \eqref{stat-equ:GP-f-wn} of $\W^j_{m+1}$ and the definition of the partial trace we have: 

\begin{align} 
\W_{m+1}^j \Biggl(  \prod_{l=j}^{j+m} \phi(t,x_\ell) \barphi(t,x'_\ell) \Biggr) 
& = -i \partial_{x_j} \W_{m}^{j} \Tr_{j+m} \Biggl(  \prod_{l=j}^{j+m} \phi(t,x_\ell) \barphi(t,x'_\ell) \Biggr) \nonumber \\
& + \; \; \; \; \sum_{k=1}^{m-1} B_{j,j+k} 
\W_{k}^{j} \otimes \W_{m-k}^{j+k} \Tr_{j+m}  \Biggl(  \prod_{l=j}^{j+m} \phi(t,x_\ell) \barphi(t,x'_\ell) \Biggr)  \nonumber\\
& = -i \partial_{x_j} \W_{m}^{j} \Biggl(  \prod_{l=j}^{j+m-1} \phi(t,x_\ell) \barphi(t,x'_\ell) \Biggr) \nonumber \\
& + \; \; \; \; \sum_{k=1}^{m-1} B_{j,j+k} 
{\color{red}{
\W_{k}^{j} 
}}
\otimes 
{\color{blue}{
\W_{m-k}^{j+k} 
}}
\Biggl(  
{\color{red}{
\prod_{l=j}^{j+k-1} \phi(t,x_\ell) \barphi(t,x'_\ell)   
}}
{\color{blue}{
\prod_{l=j+k}^{j+m-1} \phi(t,x_\ell) \barphi(t,x'_\ell) 
}}
\Biggr)  .
\nonumber
\end{align}
Using the inductive hypothesis, we then obtain
\begin{align}
 &-i \partial_{x_j}  w_m(x_j) \barphi(t,x'_j) + \sum_{k=1}^{m-1} B_{j,j+k} 
\Biggl[
w_k(x_j)\barphi(t,x'_j) \; w_{m-k}(x_{j+k}) \barphi(t,x'_{j+k}) 
\Biggr]
\label{equ:GP-f-Ww-mi-a}\\
& = -i \partial_{x_j}  w_m(x_j) \barphi(t,x'_j) 
+ \sum_{k=1}^{m-1}
w_k(x_j)\barphi(t,x'_j) \; w_{m-k}(x_j) \barphi(t,x_j) \label{equ:GP-f-Ww-mi-B}\\
& = \Biggl( -i \partial_{x_j}  w_m(x_j) 
+ \sum_{k=1}^{m-1}
w_k(x_j) w_{m-k}(x_j) \barphi(t,x_j) \Biggr) \barphi(t,x'_j) \nonumber \\
& = w_{m+1}(x_j) \barphi(t,x'_j), \label{equ:GP-f-Ww-ind}
\end{align}
where \eqref{equ:GP-f-Ww-mi-a} follows from the assumption that the identity \eqref{equ:GP-f-Ww} is satisfied for all integers $n$, $1 \leq n \leq m$, 
 \eqref{equ:GP-f-Ww-mi-B} follows from the definition of the contraction operator $B_{j,j+k}$ and the last line \eqref{equ:GP-f-Ww-ind}
follows from the recursive relation \eqref{equ:NLS-wn} satisfied by the $w_m$. 
 \end{proof}

Now we are ready to identify the conserved quantities for factorized solutions to the GP. 

\begin{prop} \label{prop:GP-f-cons}
Let $\phi(t,x) \in L^{\infty}_{t \in [0,\infty)}C^{\infty}(\mathbb R)$ be a solution to the cubic (de)focusing NLS \eqref{equ:NLS_intro}, and let 
\[
\biggl\{ \; \prod_{l=j}^{j+n-1} \phi(t,x_\ell) \barphi(t,x'_\ell)  \biggr\}_{n=1}^\infty
\]
be a sequence of factorized solutions to the GP hierarchy in  $L^{\infty}_{t \in [0,\infty)} \mathfrak{H}_{n}^{(n-1)/2 }$. 
Then for each $n \in \bN$, we have that
\begin{equation} \label{equ:GP-f-cons3} 
 \Tr \W_{n}^j \Biggl(  \prod_{l=j}^{j+n-1} \phi(t,x_\ell) \barphi(t,x'_\ell) \Biggr) = I_n(\phi),
\end{equation} 
where $ I_n(\phi)$ was defined in \eqref{equ:NLS-cons}.
In particular, this quantity is conserved in time, that is
\[
 \Tr \W_{n}^j \Biggl(  \prod_{l=j}^{j+n-1} \phi(t,x_\ell) \barphi(t,x'_\ell) \Biggr) =  \Tr \W_{n}^j \Biggl(  \prod_{l=j}^{j+n-1} \phi(0,x_\ell) \barphi(0,x'_\ell) \Biggr) .
\]
\end{prop} 

\begin{proof} 
We take the trace of the identity \eqref{equ:GP-f-Ww} and recall the expression \eqref{equ:NLS-cons} for conserved quantities for the NLS to obtain: 
\[
 \Tr \W_{n}^j \Biggl(  \prod_{l=j}^{j+n-1} \phi(x_\ell) \barphi(x'_\ell) \Biggr) 
= \int \delta (x_j - x'_j) \; w_n(x_j) \barphi(x'_j) \; dx_j
= I_n(\phi).  
\]
Since the right hand side  is constant in time, we conclude that the left hand side  is conserved in time too. 
\end{proof}

By taking $j=1$ in Proposition \ref{prop:GP-f-cons} we obtain the following corollary, 
that is equivalent to  Theorem \ref{thm:GP-f-cons-1}  stated in the Introduction.

\begin{cor} 
For  $\phi(t,x) \in L^{\infty}_{t \in [0,\infty)}C^{\infty}(\mathbb R)$ a solution to the cubic (de)focusing NLS \eqref{equ:NLS_intro}, let 
\[
\biggl\{ \; \prod_{l=j}^{j+n-1} \phi(t,x_\ell) \barphi(t,x'_\ell)  \biggr\}_{n=1}^\infty
\]
be a sequence of factorized solutions to the GP hierarchy in  $L^{\infty}_{t \in [0,\infty)} \mathfrak{H}_{n}^{(n-1)/2}$. Then for each $n \in \bN$, the quantity
\begin{equation} \label{equ:GP-f-cons4} 
 \Tr \W_{n}^1 \Biggl(  \prod_{l=1}^{n} \phi(t,x_\ell) \barphi(t,x'_\ell) \Biggr) 
\end{equation} 
is conserved in time, that is
\[
 \Tr \W_{n}^1 \Biggl(  \prod_{l=1}^{n} \phi(t,x_\ell) \barphi(t,x'_\ell) \Biggr) = \Tr \W_{n}^1 \Biggl(  \prod_{l=1}^{n} \phi(0,x_\ell) \barphi(0,x'_\ell) \Biggr).
\]
\end{cor}

\section{Conservation laws for general solutions of the GP hierarchy}  \label{sec:cons_laws_gp}

In this section we prove Theorem \ref{main}, which shows that sufficiently smooth mild solutions to the GP hierarchy admit infinitely many conserved quantities. We also prove several consequences of Theorem \ref{main}. We recall the following fact about these conserved integrals $I_n(\phi)$ for the NLS which will be needed in the sequel. The proof adapts readily from \cite[Lemma 3.2]{Zhidkov01}.
\begin{lem}\label{lem:w_bds}
For every $n \in \bN$ and $\phi$ belonging to the Sobolev space $ H^{(n-1)/2}(\R)$, 
\[
 I_n(\phi) = \int w_n(x) \overline{\phi}(x)  = \frac{1}{2} \int | \partial_x^{(n-1)/2} \phi|^2  +  p_n(\phi, \overline{\phi}, \partial_x \phi, \partial_x \overline{\phi}, \ldots, \partial_x^{(n-2)/2} \phi, \partial_x^{(n-2)/2} \overline{\phi})
\]
where $p_n$ are polynomials of degree $n+1$ and 
\[
\Biggl| \int p_n(\phi, \overline{\phi}, \ldots, \partial_x^{(n-2)/2} \phi, \partial_x^{(n-2)/2} \overline{\phi})  \Biggr| \leq C(n)\, F_n \left(\int | \partial_x^{(n-1)/2} \phi|^2 + \int | \phi|^2 \right).
\]
for a positive, continuous function $F_n(\cdot)$.
\end{lem}

\subsection{The proof of Theorem \ref{main}} \label{sec:proof_of_main}

\begin{proof} [Proof of Theorem \ref{main}]
We apply Theorem \ref{thm-uniqueness-2} and we obtain that any admissible solution $(\gamma^{(k)}(t))_{k \in \mathbb{N}}$ of the (de)focusing cubic GP hierarchy in one dimension with admissible initial data $(\gamma^{(k)}(0))_{k\in\bN} \in\frH^s$ for $s \geq \frac{1}{6}$, such that
\begin{equation} 
    \gamma^{(k)}(0) = \int d\mu(\phi_0)(|\phi_0\rangle\langle\phi_0|)^{\otimes k} 
    \;\;\;\;\;\;\forall k\in\bN\,,
\end{equation} 
where $\mu$ is a Borel probability measure supported  either  on the unit sphere or on the unit ball in $L^2(\R)$, and invariant under multiplication of  $\phi \in L^2(\R)$ by complex numbers of modulus one, can be written as

\begin{equation} \label{eq-GPdeF-NLS-anyD}
    \gamma^{(k)}(t) = \int d\mu(\phi_0)(|\phi(t,x)\rangle\langle \phi(t,x)|)^{\otimes k} 
    \;\;\;\;\;\;\forall k\in\bN\,,
\end{equation} 
where $\phi(t,x)$ is a solution to the cubic (de)focusing NLS \eqref{equ:NLS_intro} with initial data $\phi_0$, for $t\in[0,\infty)$. Consequently, for any $n \in \bN$, $1 \leq j \leq n$ such that $k \geq j+n-1$ by Proposition \ref{prop:ops_defined} at each fixed time (and Remark \ref{rmk:gen_defn}) we have
\[
\Tr \W_{n}^j \gamma^{(k)}(t) = \Tr \int d\mu(\phi_0) \W_n^j (|\phi(t,x)\rangle\langle \phi(t,x)|)^{\otimes k} .
\]
Since the trace is a continuous linear operator, we may once again bring it inside the integral using Hille's theorem by observing that
\[
\Tr  \W_n^j \bigl(|\phi(t,x)\rangle\langle \phi(t,x)|\bigr)^{\otimes k} =  I_n(\phi(t,x))
\]
is a complex-valued function which is Lebesgue integrable over the unit sphere by Lemma \ref{lem:w_bds}, and hence Bochner integrable with codomain $\R$. We obtain
\begin{align}
&\Tr \int d\mu(\phi_0) \W_n^j (|\phi(t,x)\rangle\langle \phi(t,x)|)^{\otimes k} \\
&= \int d\mu(\phi_0) \Tr  \W_n^j (|\phi(t,x)\rangle\langle \phi(t,x)|)^{\otimes k} \\
&= \int d\mu(\phi_0) I_n(\phi(t,x)),
\end{align}
where in the last equality we have used Proposition \ref{prop:GP-f-cons}. We conclude that
\[
\int d\mu(\phi_0) I_n(\phi(t,x)) = \int d\mu(\phi_0) I_n(\phi(0,x)) = \int d\mu(\phi_0) I_n(\phi_0),
\]
since $I_n(\phi(t,x))$ is conserved in time for solutions of the cubic NLS \eqref{equ:NLS_intro}. Now, using Proposition \ref{prop:GP-f-cons} once again, we can argue as above to obtain
\begin{align*}
\int d\mu(\phi_0) I_n(\phi_0) = \int d\mu(\phi_0) \Tr \W_n^j  (|\phi_0\rangle\langle \phi_0|)^{\otimes k}  = \Tr \W_n^j \int d\mu(\phi_0) (|\phi_0\rangle\langle \phi_0|)^{\otimes k}  = \Tr \W_n^j \gamma^{(k)}(0),
\end{align*}
which concludes our proof.
\end{proof} 


\subsection{Further applications}

In this last subsection, we collect some results which are consequences of Theorem \ref{main}.  From our Definition \ref{def:GP-f-w}  of the operators $\W_n^j$ and induction, we first prove that the structure of the operators is as follows:

\begin{cor} \label{cor:op_form}
For the operators $\W_n^j$ defined above, one has
\[
\W_{n}^j =  (-i\partial_{x_j})^{n-1} \Tr_{j+n-1} + P_n(B_{i_1, j_1}, \ldots, B_{i_n, j_n} , \partial_{x_{1}}, \ldots, \partial_{x_n}) \Tr_{j+n-1},
\]
for $P_n$ is a polynomial of degree $n-1$ and 
\[
P_n(B_{i_1, j_1}, \ldots, B_{i_n, j_n} , \partial_{x_{1}}, \ldots, \partial_{x_n}) : \mathfrak{h}_{n}^{(n-1)/2 + \alpha} \to \mathfrak{h}_{1}^{\alpha}.
\]
\end{cor}
\begin{rmk}
This corollary is reminiscent of the properties of the conservation laws for the cubic NLS stated in Lemma \ref{lem:w_bds}. In particular, the lower order terms are bounded linear operators from an appropriate Sobolev type space to an $L^2$ type space.
\end{rmk}
%
%

Finally, we return to the statement made in Remark \ref{rmk:gen_cons}.

\begin{prop} \label{prop:higher_cons}
Let $\{n_1, \ldots, n_k\}$ be any finite sequence of indices with $n_i \in \bN$, and let $\{ \gamma^{(N)} \}_{N} \in\frH^{\infty}$ be a solution to the GP hierarchy. Then the quantity
\[
\Tr\bigl(\W_{n_1}^1 \otimes \W_{n_2}^{1 + n_1} \otimes \W_{n_3}^{1 + n_1 + n_3}  \otimes \ldots  \otimes \W_{n_k}^{1 + n_1+ \ldots+ n_{k-1}} \gamma^{(n_1 + \ldots + n_k)}(t) \bigr)
\]
is conserved in time.
\end{prop}
\begin{proof}
We let $m = \max_i (n_i-1)/2$, and denote the upper indices by $j_1, \ldots, j_k$. We note that each $\W_{n_i}^{j_i}: \mathfrak{h}_{n_i}^{(m-1)/2} \to \mathfrak{h}_{1}^{(m-1)/2 - (n_i - 1)/2} \subseteq \mathfrak{h}_{1}^{0}$ is a bounded linear operator. We use the strong quantum de Finetti theorem and Theorem \ref{thm-uniqueness-2} to write
\[
\gamma^{(n_1 + \ldots + n_k)}(t) = \int d \mu(\phi) (|\phi(t,x)\rangle\langle \phi(t,x)|)^{\otimes (n_1 + \ldots + n_k)}.
\]
Once again, we let
\[
f_k: H^s \to \mathfrak{h}_k^s, \qquad \phi \mapsto \prod_{i=1}^k \phi(x_i) \overline{\phi}(x_i'),
\]
and note that $f_{n_1+ \ldots + n_k} = f_{n_1} f_{n_2} \ldots f_{n_k}$. We can write
\[
\gamma^{(n_1 + \ldots + n_k)}(t) = \int d \mu(\phi) f_{n_1}(\phi(t,x)) f_{n_2} (\phi(t,x)) \ldots f_{n_k} (\phi(t,x)).
\]
We now observe that
\[
\W_{n_1}^{j_1}  f_{n_1}(\phi(t,x_{j_1})) \cdot  \W_{n_2}^{j_2} f_{n_2} (\phi(t,x_{j_2})) \cdots \W_{n_k}^{j_k} f_{n_k} (\phi(t,x_{j_k})) \in \mathfrak{h}_{k}^0,
\]
and by the definition of the norm on $\mathfrak{h}_{k}^0$, we can bound
\[
\bigl\| \W_{n_1}^{j_1}  f_{n_1}(\phi(t,x_{j_1})) \cdot  \W_{n_2}^{j_2} f_{n_2} (\phi(t,x_{j_2})) \cdots \W_{n_k}^{j_k} f_{n_k} (\phi(t,x_{j_k}))\bigr\|_{\mathfrak{h}_{k}^0} \leq \prod_{i=1}^k \bigl\|\W_{n_1}^{j_1}  f_{n_1}(\phi(t,x_{j_1}))\bigr\|_{H^m}.
\]
%
Hence, arguing as in the proof of Proposition \ref{prop:ops_defined}, relying on Theorem \ref{thm:l1_boch_int} and Theorem \ref{thm:Hille}, we obtain
\begin{align*}
&\W_{n_1}^{j_1} \otimes \W_{n_2}^{j_2}  \otimes \ldots  \otimes \W_{n_k}^{j_k} \gamma^{(n_1 + \ldots + n_k)}(t) \\
&=  \int d \mu(\phi) \W_{n_1}^{j_1}  f_{n_1}(\phi(t,x_{j_1})) \cdot  \W_{n_2}^{j_2} f_{n_2} (\phi(t,x_{j_2})) \cdots \W_{n_k}^{j_k} f_{n_k} (\phi(t,x_{j_k})) .
\end{align*}
Now, using a similar Hille's theorem argument for the trace as in the proof of Theorem \ref{main}, we have
\begin{align*}
& \Tr \W_{n_1}^{j_1} \otimes \W_{n_2}^{j_2}  \otimes \ldots  \otimes \W_{n_k}^{j_k} \gamma^{(n_1 + \ldots + n_k)}(t) \\
& = \int d \mu(\phi_0)  \Tr\bigl( \W_{n_1}^{j_1}  f_{n_1}(\phi(t,x)) \bigr) \cdot \Tr \bigl( \W_{n_2}^{1 + n_1} f_{n_2} (\phi(t,x)) \bigr)\cdots \Tr \bigl(\W_{n_k}^{1 + n_1+ \ldots+ n_{k-1}} f_{n_k} (\phi(t,x))\bigr)\\
& =  \int d \mu(\phi_0)  I_{n_1}(\phi_0)  I_{n_2}(\phi_0) \cdots I_{n_k}(\phi_0),
\end{align*}
where we have used in the last step that the $I_{n_k}(\phi(t))$ are conserved for solutions of the cubic NLS. We can observe that the final expression is constant, to conclude that these quantities are conserved. Finally, as above, we obtain that by unfolding this same argument,
\begin{align}
& \int d \mu(\phi_0)  I_{n_1}(\phi_0)  I_{n_2}(\phi_0) \cdots I_{n_k}(\phi_0) \\
&= \Tr \W_{n_1}^1 \otimes \W_{n_2}^{1 + n_1} \otimes \W_{n_3}^{1 + n_1 + n_3}  \otimes \ldots  \otimes \W_{n_k}^{1 + n_1+ \ldots+ n_{k-1}} \gamma^{(n_1 + \ldots + n_k)}(0),
\end{align}
as desired.
\end{proof}

\begin{rmk}
In particular, Proposition \ref{prop:higher_cons} recovers the result of \cite{CP14gwp}. Moreover, any finite linear combination of the type of operators defined in Proposition \ref{prop:higher_cons} will yield conserved quantities for the GP hierarchy. 
\end{rmk}

\bibliographystyle{myamsplain}
\bibliography{refs}

\end{document}